\documentclass[12pt,a4paper,reqno]{amsart}
\usepackage[a4paper,left=30mm,right=30mm,top=36mm,bottom=36mm]{geometry}
\theoremstyle{definition}
\newtheorem{dfntn}{Definition}[section]
\newtheorem{rmrk}[dfntn]{Remark}
\newtheorem{thrm}[dfntn]{Theorem}
\newtheorem{lmm}[dfntn]{Lemma}
\newtheorem{prop}[dfntn]{Proposition}
\newtheorem{cor}[dfntn]{Corollary}
\newtheorem{eg}[dfntn]{Example}
\newtheorem{notation}[dfntn]{Notation}
\numberwithin{equation}{section}

\usepackage{amssymb}
\usepackage{amsmath}
\usepackage{graphicx}
\usepackage{float}
\usepackage[utf8]{inputenc}
\usepackage[english]{babel}
\usepackage{subfig}
\usepackage{mathtools}
\newcommand{\R}{\mathbb{R}}
\newcommand{\N}{\mathbb{N}}
\renewcommand{\div}{\mathrm{div}}

\usepackage{setspace}  
\usepackage{accents}   
\usepackage{cancel}    
\allowdisplaybreaks    
\usepackage[colorlinks, bookmarksnumbered, bookmarks, linkcolor=red]{hyperref}
\usepackage{upgreek}

\newcommand{\op}[1]{#1}               
\newcommand{\hilbert}[1]{#1}          

\DeclareMathOperator{\grad}{grad}     
\DeclareMathOperator{\Grad}{Grad}
\DeclareMathOperator{\Div}{Div}
\DeclareMathOperator{\Rot}{Curl}
\DeclareMathOperator{\rot}{curl}
\DeclareMathOperator{\skw}{skw}
\DeclareMathOperator{\adjugate}{adj}
\DeclareMathOperator{\divcirc}{\accentset{\circ}{\div}}
\DeclareMathOperator{\gradcirc}{\accentset{\circ}{\grad}}
\DeclareMathOperator{\Gradcirc}{\accentset{\circ}{\Grad}}
\DeclareMathOperator{\Divcirc}{\accentset{\circ}{\Div}}
\DeclareMathOperator{\Rotcirc}{\accentset{\circ}{\Rot}}

\DeclareMathOperator{\Span}{span}

\begin{document}

\onehalfspacing

\title[Revisiting Picard's proof of Picard--Weber--Weck selection theorem]{Revisiting Picard's proof of the Picard--Weber--Weck selection theorem}

\author[Y.-S. Lim]{Yi-Sheng Lim}
\address[Yi-Sheng Lim]{Department of Mathematics, Texas A{\&}M University, College Station, TX 77843, USA.}
\email{yishenglimysl@tamu.edu}

\author[T. Sprekeler]{Timo Sprekeler}
\address[Timo Sprekeler]{Department of Mathematics, Texas A{\&}M University, College Station, TX 77843, USA.}
\email{timo.sprekeler@tamu.edu}

\author[M. Waurick]{Marcus Waurick}
\address[Marcus Waurick]{Institut f\"{u}r Angewandte Analysis, Technische Universit\"{a}t Bergakademie Freiberg, Germany.}
\email{marcus.waurick@math.tu-freiberg.de}

\date{\today}

\begin{abstract} Revisiting the rationale provided by Picard in the seminal paper \cite{Picard1984}, we provide an independent, self-contained proof of the Maxwell compactness property for weak Lipschitz domains in the Euclidean setting, detouring technical complications like the differential forms setting, Gaffney's inequality for smooth domains, Calderon's extension theorem or regularity theory for PDEs, or systems thereof, on smooth domains that are used in this context. As a by-product we provide an independent proof of the classical Gaffney estimate for cubes using only $L^2$-completeness of the Fourier bases.
\end{abstract}

\maketitle

\section{Introduction}
This paper is concerned with a by-now classical compactness result for differential forms going back to several articles by Weck \cite{Weck1974}, Weber \cite{Weber1980}, and Picard \cite{Picard1984}. The main contribution of all of these results is that, for the sake of the argument expressed in dimension $d=3$, the embedding
\begin{equation}\label{eq:intro1}
   H_0(\rot;\Omega)\cap H(\div;\Omega)\hookrightarrow L^2(\Omega)^3
\end{equation}
is compact for a broad class of rather nonsmooth bounded domains $\Omega\subseteq \R^3$. Previously, $\Omega$ was required to be at least a $C^{1,1}$-domain in which case the continuous embedding
\[
   H_0(\rot;\Omega)\cap H(\div;\Omega) \hookrightarrow H^1(\Omega)^3,
\]
provided by a Gaffney-type estimate together with the classical Rellich--Kondrachov selection theorem, would yield the compactness statement for \eqref{eq:intro1}; see, e.g., \cite[Theorem 8.6]{Leis1986} for a mere historical reference and a proof for $C^5$-domains. Weck \cite{Weck1974} was the first to bypass Gaffney's estimate and still recover compactness of \eqref{eq:intro1} for certain singular domains with a bootstrapping argument in the dimension of the domain. Later, in \cite{Weber1980}, Weber provided a perspective on the compactness of \eqref{eq:intro1} with the main technical ingredient of Calderon's extension theorem. We note in passing that this is also the circle of ideas in \cite{BPS16}, tied together with regular potentials. Here we shall revisit the seminal work \cite{Picard1984} of Picard and the corresponding proof of the Maxwell compactness property. The key for \cite{Picard1984} to work is two-fold: the tailored setting provided by differential forms and a Gaffney estimate for the cube. Even though the proof in \cite{Picard1984} is very elegant and comparatively simple, the main ingredients might be rather inaccessible to many due to either the rather involved machinery of differential forms or the elusive Gaffney estimate. Even though well-known, the proof of the latter inequality appears to be technically demanding in many textbooks and articles, being based on an integration-by-parts formula involving the curvature of the boundary of the underlying domain; see, e.g., \cite{ABD98,Cos91,Gri85,Pau19,Sar82}. In any case some more theory needs to be appended, which is rooted in regularity theory for solutions of PDEs, or systems thereof, on smooth/convex domains; see, e.g., the proof in \cite{Mitrea2001} or the rather recent work \cite{PaulyWaurick2026}. 

Thus, an independent proof without a rather strong background in both differential forms and regularity theory for PDEs appears to be missing in the literature. Hence, here we provide a proof of the compactness of \eqref{eq:intro1} for weak Lipschitz domains without resorting to the Calderon extension theorem, regularity theory for PDE systems or a demanding integration by parts formula. In fact, by using functional analytic (`soft') arguments mainly, we reduce the amount of explicit computations to a minimum and, thus, only by using $L^2$-completeness of the Fourier basis functions, we shall establish Gaffney's estimate on the cube and then close the elementary argument for the compactness statement via explicit transformation rules for the Euclidean case, not needing to resort to the calculus of differential forms. The argument being essentially the same for any dimension, we thus provide an independent proof of the following statement; for the explicit definition of `weak Lipschitz', we refer to the main body of the present text.

\begin{thrm}[Picard--Weber--Weck selection theorem]\label{thrm:mtintro} Let $\Omega\subseteq \R^d$ be a bounded, weak Lipschitz domain. Then, the embedding
\[
    H_0(\Rot;\Omega)\cap H(\div;\Omega)\hookrightarrow L^2(\Omega)^d
\]
is compact.
\end{thrm}
\begin{rmrk}
(a) The operator $\Rot$ is the (scaled) skew-symmetric part of the Jacobian and $\Rotcirc$ is the same with homogeneous boundary conditions; see Definition \ref{def: Rotop}. Similarly, $\div$ is the divergence operator with maximal $L^2$-domain and $\divcirc$ is the corresponding operator with homogeneous boundary conditions; see Definition \ref{defn:operator_domains}. We write $H_0(\Rot;\Omega):= D(\Rotcirc_\Omega)$ and $H_0(\div;\Omega):= D(\divcirc_\Omega)$.

(b) The corresponding statement with the other boundary conditions, i.e., compactness of
\[
    H(\Rot;\Omega) \cap H_0(\div;\Omega)\hookrightarrow L^2(\Omega)^d,
\]
is also true and follows by the same reasoning as the proof of Theorem \ref{thrm:mtintro}. We shall comment on the necessary adjustments for this case in the places, where they are necessary.
\end{rmrk}

A major application of Gaffney's inequality is in the analysis of mixed finite element methods for Hamilton--Jacobi--Bellman \cite{GSS21,GS19,HMS26,Spr24} and Fokker--Planck--Kolmogorov equations \cite{Spr26,SSZ25}.

This paper is organized as follows. In Section \ref{Sec:2}, we go straight to proving the selection theorem. This will be subject to a Gaffney (in)equality for cubes and a transformation theorem, both of which are proven in the subsequent sections. For the former, some functional analytic results need to be established first; see Section \ref{sect:abstract_framework}. In Section \ref{sect:gaffney_cube}, we present a self-contained proof of the Gaffney result for cubes. Finally, in Section \ref{sec:compatibility_proof}, we provide a proof of the transformation theorem, which concludes the article.

\section{Picard--Weber--Weck selection theorem}\label{Sec:2}

In this section, we provide the desired proof of the selection theorem subject to some auxiliary results to be proven later. We will require the following differential operators.

\begin{dfntn}[gradient and divergence operators]\label{defn:operator_domains}

Let $\Omega\subseteq \R^d$ be open. We define the differential operators
\begin{align*}
    \grad_c&: C_c^\infty(\Omega)\subseteq L^2(\Omega)\rightarrow L^2(\Omega)^d,\quad v\mapsto (\partial_1 v,\ldots,\partial_d v)^\top,\\ 
    \div_c&: C_c^\infty(\Omega)^d\subseteq L^2(\Omega)^d\rightarrow L^2(\Omega),\quad (v_1,\ldots,v_d)^\top\mapsto \sum_{k = 1}^d \partial_k v_k,\\
    \Grad_c&: C_c^\infty(\Omega)^d\subseteq L^2(\Omega)^d\rightarrow L^2(\Omega)^{d\times d},\quad (v_1,\ldots,v_d)^\top\mapsto (\partial_j v_i)_{1\leq i,j\leq d},\\
    \Div_c&: C_c^\infty(\Omega)^{d\times d}\subseteq L^2(\Omega)^{d\times d}\rightarrow L^2(\Omega)^{d},\quad (v_{ij})_{1\leq i,j\leq d}\mapsto (\sum_{j = 1}^d \partial_j v_{ij})_{1\leq i\leq d}.
\end{align*}
We set
\begin{align*}
    \gradcirc := \overline{\grad_c},\qquad  \divcirc := \overline{\div_c},\qquad \Gradcirc := \overline{\Grad_c},\qquad  \Divcirc := \overline{\Div_c},
\end{align*}
and
\begin{align*}
    \grad := -(\divcirc)^\ast,\quad \div := -(\gradcirc)^\ast,\quad \Grad := -(\Divcirc)^\ast,\quad \Div := -(\Gradcirc)^\ast.
\end{align*}
\end{dfntn}

\begin{dfntn}[curl-type operator]\label{def: Rotop}
    Let $\Omega\subseteq \R^d$ be open. We define the differential operator
\begin{align*}
    \Rot_c &:C_c^\infty(\Omega)^d\subseteq L^2(\Omega)^d\rightarrow L^2(\Omega)^{d\times d},\quad  v\mapsto \skw(\Grad_c   v),
\end{align*}
where
\begin{align*}
    \skw:L^2(\Omega)^{d\times d}\rightarrow L^2(\Omega)^{d\times d},\quad V\mapsto \frac{1}{\sqrt{2}}(V-V^\top),
\end{align*}
and we set $\Rotcirc := \overline{\Rot_c}$ and $\Rot := -(\Divcirc \skw)^\ast$.
\end{dfntn}

\begin{notation}
    For a linear operator $\op A$, we shall write $D(\op A)$, $N(\op A)$, and $R(\op A)$ for its domain, kernel, and range, respectively. 
\end{notation}

\begin{rmrk}\label{rmk:adjoints_of_key_operators}
    The adjoint of the operator $\Rotcirc$ is given by $(\Rotcirc)^\ast = -\Div \skw$.
\end{rmrk}

\begin{dfntn}
    Let $\Omega\subseteq \R^d$ be open. We use the following standard notation for Sobolev spaces:
    \begin{alignat*}{3}
        H^1(\Omega) &:= D(\grad),\qquad \qquad\;\; H^1_0(\Omega) &&:= D(\gradcirc),\\
        H(\div;\Omega) &:= D(\div),\qquad \quad \;\,H_0(\div;\Omega) &&:= D(\divcirc),\\
        H(\Rot;\Omega) &:= D(\Rot),\qquad \; H_0(\Rot;\Omega) &&:= D(\Rotcirc).
    \end{alignat*}
    We equip the domains with their respective graph norms, e.g., $(D(\Rot),\|\cdot \|_{H(\Rot;\Omega)})$ with $\| \cdot \|_{H(\Rot;\Omega)} := (\| \cdot \|_{L^2(\Omega)^d}^2 + \| \Rot( \cdot) \|_{L^2(\Omega)^{d\times d}}^2)^{1/2}$.
\end{dfntn}

\begin{dfntn}\label{defn:weak_lipschitz}
    A set $\Omega \subseteq \R^d$ is a \textbf{weak Lipschitz domain} if the boundary $\Gamma = \partial \Omega$ is a Lipschitz submanifold. That is, there exists a finite open cover $U_1, \ldots, U_K \subseteq \R^d$ of $\Gamma$ and maps $\varphi_k : U_k \rightarrow (0,1)^{d-1}\times (-1,1)$ such that for $k \in \{ 1,\ldots, K\}$, $\varphi_k$ is bi-Lipschitz, bounded, regular, and satisfies
    \begin{align*}
        \varphi_k(U_k \cap \Omega ) = (0,1)^{d} =: Q.
    \end{align*}
    That is, $\varphi_k \in C^{0,1}_{b}(\overline{U_k}, [0,1]^{d-1}\times [-1,1])$, $\varphi_k^{-1} \in C^{0,1}_{b}([0,1]^{d-1}\times [-1,1], \overline{U_k})$, and there exists $c>0$ such that the Jacobian $J_{\varphi_k} := \Grad(\varphi_k)$ of $\varphi_k$ satisfies
    \begin{align*}
        \det{J_{\varphi_k}(x)} \geq c \quad \text{for $a.e.$ $x \in U_k$.}
    \end{align*}
\end{dfntn}

Now, the Picard--Weber--Weck selection theorem reads as follows.

\begin{thrm}\label{thrm:mtsec2} Let $\Omega\subseteq \R^d$ be a bounded, weak Lipschitz domain. Then,
\[
  H_0(\Rot;\Omega)\cap H(\div;\Omega) \hookrightarrow L^2(\Omega)^d
\]
compactly.
\end{thrm}

The proof will be carried out at the end of this section. It will be deduced eventually from Gaffney's inequality for cubes, which results in the following regularity statement.

\begin{thrm}\label{thrm:gaffcub} Let $Q=(0,1)^d$. Then
   \[
  H_0(\Rot;Q)\cap H(\div;Q) \hookrightarrow H^1(Q)^d.
\]
In particular,
\[
  H_0(\Rot;Q)\cap H(\div;Q) \hookrightarrow L^2(Q)^d
\]
compactly.
\end{thrm}
\begin{proof}
The proof will be carried out in the subsequent sections, see Theorem \ref{thm:gaffney_cube}. Note that the last statement follows from the first invoking the classical Rellich--Kondrachov selection theorem.
\end{proof}

In order to achieve a proof of Theorem \ref{thrm:mtsec2}, we require a compatibility formula for bi-Lipschitz transformations of global Lipschitz domains with the operators $\div$ and $\Rot$. We recall the following definition from \cite{PaulyWaurick2026}:

\begin{dfntn}\label{defn:global_lipschitz}
    A set $\Omega \subseteq \R^d$ is a \textbf{global (strong) Lipschitz domain} if there exist an open cube $Q = (0,1)^d$ and a Lipschitz function $\Phi: \R^d \rightarrow \R^d $ that maps $Q$ bijectively to $\Omega$, and its restriction to $Q$, still denoted by
    \begin{align*}
        \Phi : Q \rightarrow \Omega,
    \end{align*}
    is bi-Lipschitz, bounded, and regular. We refer to $\Phi:Q\rightarrow \Omega$ as an \textbf{admissible} map for $\Omega$.
\end{dfntn}

The compatibility theorem now reads as follows. For this note that for the considered differential operators, we shall write for instance $\Rotcirc_\Omega$ and $\Rotcirc_Q$, to clarify that it acts on vector fields defined on $\Omega$ and $Q$ respectively. When this is clear from the context, we shall keep the notation compact and simply write for instance $\Rotcirc$. This applies verbatim to $\div$.

\begin{prop}\label{prop:rot_div_comp} Let $\Omega\subseteq \R^d$ be a global Lipschitz domain and $\Phi\colon Q\to \Omega$ be an admissible transformation. Writing $J_\Phi := \Grad(\Phi)$ and $\tilde{u} := u \circ \Phi$, the following assertions hold.

    (a)
    Let $  u \in H_0(\Rot;\Omega)$. Then 
    $J_\Phi^\top \widetilde{  u} \in H_0(\Rot;Q)$ and
    \begin{equation*}
        \Rotcirc_Q{(J_\Phi^\top \widetilde{  u})} 
        =  J_\Phi^\top\,\widetilde{\Rotcirc_\Omega   u}\,J_\Phi =: ( J_\Phi^\top \otimes J_\Phi^\top )\widetilde{\Rotcirc_\Omega   u}.
    \end{equation*}

    (b)  Let $  u \in H(\div;\Omega)$. Then $(\adjugate J_\Phi) \widetilde{  u} \in H(\div;Q)$ and
    \begin{align*}
        \div_Q{( (\adjugate J_\Phi) \widetilde{  u})} 
        = (\det{J_\Phi})\widetilde{\div_\Omega   u}.
    \end{align*}
\end{prop}

Proposition \ref{prop:rot_div_comp} shall be proven in Section \ref{sec:compatibility_proof} as Proposition \ref{prop:rot_composition} for (a) and Proposition \ref{prop:div_composition} for (b). Before we get to the proof of Theorem \ref{thrm:mtsec2}, we will need a little observation, which is a special case of Helga's theorem, see \cite[Theorem 9.1]{Waurick2025}. Here, we write $L(H)$ for the collection of bounded linear operators on a Hilbert space $H$.

\begin{thrm}\label{thrm:abstractcompact} 
    Let $H_0,H_1,H_2$ be Hilbert spaces, $A_0\colon D(A_0)\subseteq H_0\to H_1$ and $A_1\colon D(A_1)\subseteq H_1\to H_2$ be densely defined closed linear operators satisfying $R(A_0)\subseteq N(A_1)$. Moreover, let $a \in L(H_1)$ satisfy $\Re a\coloneqq \frac{1}{2}(a+a^*)\geq c$ for some $c>0$ in the sense of positive definiteness. 

    If $D(A_1)\cap D(A_0^*)\hookrightarrow H_1$ is compact, then $D(A_1)\cap D(A_0^* a)\hookrightarrow H_1$ is compact.
\end{thrm}

\begin{proof}
    Let $(E_n)_n$ be a weak nullsequence in $D(A_1)\cap D(A_0^* a)$. That is, $E_n\rightharpoonup 0$ in $H_1$, $A_1 E_n\rightharpoonup 0$ in $H_2$, and $A_0^\ast aE_n \rightharpoonup 0$ in $H_0$. It suffices to show that $E_n\to 0$ strongly in $H_1$. For this, we note that the assumption implies that $D(A_1)\cap N(A_1)^\bot\hookrightarrow H_1$ and $D(A_0^*)\cap N(A_0^*)^\bot\hookrightarrow H_1$ compactly, where we view for instance, $N(A_1)^\perp$ as a subspace of $(D(A_0^\ast), \|\cdot \|_{A_0^\ast})$.
    
    Let $\pi_1\colon H_1\to H_1$ be the orthogonal projection onto $N(A_1)^{\bot}$. Then $A_1 \pi_1 E_n = A_1 E_n \stackrel{H_2}{\rightharpoonup} 0$, $\pi_1 E_n \stackrel{H_1}{\rightharpoonup} 0$, and $A_0^\ast \pi_1 E_n \equiv 0$, and thus we infer by compact embedding that $\pi_1 E_n \to 0$  strongly in $H_1$. Similarly, if $\pi_0$ denotes the orthogonal projection on $N(A_0^*)^\bot$, we deduce that $\pi_0 a E_n \to 0$ strongly in $H_1$.
    Next, consider the orthogonal decomposition $H_1 = \overline{R(A_0)} \oplus \left( N(A_1) \cap N(A_0^\ast) \right) \oplus \overline{R(A_1^\ast)}$, and observe that $\pi_2 := I - \pi_1 - \pi_0$ is the orthogonal projection onto $N(A_1) \cap N(A_0^\ast)$. We claim that $N(A_1) \cap N(A_0^\ast)$ is finite-dimensional. Indeed, consider the composition of embeddings between Banach spaces
    \begin{align*}
        (N(A_1) \cap N(A_0^\ast), \| \cdot \|_{H_1}) \stackrel{i}{\hookrightarrow}
        (D(A_1) \cap D(A_0^\ast), \| \cdot \|_{D(A_1) \cap D(A_0^\ast)}) \stackrel{j}{\hookrightarrow}
        (H_1, \| \cdot \|_{H_1}).
    \end{align*}
    Observe that $\| x \|_{H_1} = \| x \|_{D(A_1) \cap D(A_0^\ast)}$ when $x \in N(A_1) \cap N(A_0^\ast)$. Thus $i$ is continuous, and so $ji$ is compact. Further, $ji$ restricts to the identity map on $(N(A_1) \cap N(A_0^\ast), \| \cdot \|_{H_1})$. This proves the claim.

    Thus $\pi_2$ is finite-rank, and in particular compact. We now obtain
    \begin{align*}
      c\langle E_n,E_n\rangle_{H_1} 
      & \leq \Re \langle aE_n,E_n \rangle_{H_1} \\
      & =\Re \langle \pi_0(aE_n), \pi_0 E_n \rangle_{H_1}
      +\Re \langle \pi_1(aE_n), \pi_1 E_n \rangle_{H_1}
      +\Re \langle \pi_2(aE_n), \pi_2 E_n \rangle_{H_1} \\
      &\rightarrow 0,
    \end{align*}
    where we used that the scalar product of a weakly convergent sequence with a strongly convergent one converges to the scalar product of the respective limits of the two sequences (see e.g.~\cite[Proposition~1.19]{cioranescu_donato}).
    The above computation shows $\|E_n\|_{H_1}^2\to 0$, which shows the assertion.
\end{proof}

Finally, we may now prove the main theorem of the present section.

\begin{proof}
    [Proof of Theorem \ref{thrm:mtsec2}]
\textbf{Step 1.} $\Omega$ is a global Lipschitz domain. 

 Let $(u_n)_n$ be bounded in $ H_0(\Rot;\Omega)\cap H(\div;\Omega)$. For any $n\in \N$ define $\widetilde{u}_n :=  u\circ \Phi$, where $\Phi\colon Q\to \Omega$ is bi-Lipschitz. 
  
  Then, by Proposition \ref{prop:rot_div_comp}, we have that
  \begin{align*}
      \div_Q ((\adjugate J_\Phi) \widetilde{u}_n) &= (\det J_\Phi)\widetilde{
      \div_\Omega u_n},\\  \Rotcirc_Q( J_\Phi^\top \widetilde{ u}_n) &= (J_\Phi^\top \otimes J_\Phi^\top)\widetilde{\Rotcirc_\Omega u_n}
  \end{align*} 
  for any $n\in \N$. In particular, for $\widetilde{v}_n \coloneqq  J_\Phi^\top \widetilde{ u}_n$ we infer for $a\coloneqq (\det J_\Phi )J_\Phi^{-1}J_\Phi^{-\top}$ (which is strictly positive definite, as $\Phi$ is a bi-Lipschitz homeomorphism) that
    \begin{align*}
    \div_Q a \widetilde{v}_n &=    \div_Q((\adjugate J_\Phi) \widetilde{u}_n)  = (\det J_\Phi)\widetilde{
      \div_\Omega u_n}, \\ \Rotcirc_Q \widetilde{v}_n &= (J_\Phi^\top \otimes J_\Phi^\top)\widetilde{\Rotcirc_\Omega u_n}.
  \end{align*}
  By assumption, it follows that $(\widetilde{v}_n)_n$ is bounded in $D(\Rotcirc_Q)\cap D(\div_Q a)$. By Theorem \ref{thrm:gaffcub}, we have that $D(\Rotcirc_Q) \cap D(\div_Q)\hookrightarrow L^2(Q)^d$ compactly. Hence, by Theorem \ref{thrm:abstractcompact}, applied to $A_0=\gradcirc$ and $A_1=\Rotcirc$, we have that $D(\Rotcirc_Q) \cap D(\div_Q a)\hookrightarrow L^2(Q)^d$ compactly. As a consequence, $(\widetilde{v}_n)_n$ has an $L^2(Q)^d$-convergent subsequence and, thus, so does $(u_n)_n$ proving the assertion for the global Lipschitz case.

  \textbf{Step 2.} $\Omega$ is a weak Lipschitz domain.
  
  Let $\mathcal{U}$ be an open cover consisting of the open sets $U_1,\ldots , U_K$ covering $\partial\Omega$ and for a fixed $\varepsilon>0$ small enough, the family $x+\varepsilon (0,1)^d \in \mathcal{U}$, $x\in \Omega$, with $\operatorname{dist}(\partial\Omega,x+\varepsilon (-1/2,1/2)^d)>0$. Then $\mathcal{U}$ is an open cover of $\overline{\Omega}$ and we find a finite subcover $\mathcal{F}\subseteq \mathcal{U}$. Next, choose $(\phi_F)_{F\in \mathcal{F}}$, a $C^\infty$-partition of unity subordinate to $\mathcal{F}$.
  Next, note that for each $F \in \mathcal{F}$, by the product rule, the mapping
  \[
  H_0(\Rot;\Omega)\cap H(\div;\Omega)\to H_0(\Rot;F) \cap H(\div;F),\qquad u\mapsto \phi_F u 
  \] is continuous
  and $F$ is a global Lipschitz domain.
  Thus, any bounded sequence $(u_n)_n$ in $H_0(\Rot;\Omega) \cap H(\div;\Omega)$ admits a subsequence such that $(\phi_F u_n )_n$ is $L^2(\Omega)^d$-convergent by Step 1.
  As $(\phi_F)_F$ is a partition of unity of $\Omega$, it follows that the chosen subsequence is $L^2(\Omega)^d$ convergent and the claim follows.
\end{proof}

The statements left to prove are now Theorem \ref{thrm:gaffcub} and Proposition \ref{prop:rot_div_comp}. We shall start with the former for which we need some prerequisites in the functional analysis of Hilbert complexes, see also \cite{BPS16,PaulyWaurick2026}
 and the references therein.
 
\section{Abstract functional analytic framework}\label{sect:abstract_framework}

\noindent
In this section we shall assume the following setup:
\begin{itemize}
    \item $\hilbert H_0$, $\hilbert H_1$, and $\hilbert H_2$ are Hilbert spaces, which may be real or complex.
    \item $\op A_0: D(\op A_0) \subseteq \hilbert H_0 \rightarrow \hilbert H_1$ and $\op A_1: D(\op A_1) \subseteq \hilbert H_1 \rightarrow \hilbert H_2$ are densely defined, closed operators, such that
    \begin{align}\label{eqn:range_ker_condition}
        \overline{R(\op A_0)} \subseteq N(\op A_1).
    \end{align}
\end{itemize}

We make a few preliminary remarks. First, recall that $\op A_0$ and $\op A_1$ induce the orthogonal decompositions
\begin{align*}
    \hilbert H_1 
    = \overline{R(\op A_0)} \oplus N(\op A_0^\ast), \qquad 
    \hilbert H_1
    = N(\op A_1) \oplus \overline{R(\op A_1^\ast)},
\end{align*}
which, in view of \eqref{eqn:range_ker_condition}, implies that
\begin{align*}
    \overline{R(\op A_1^\ast)} \subseteq N(\op A_0^\ast).
\end{align*}

Second, for application of the abstract results of the present section, the reader may keep in mind the following example:

\begin{eg}\label{eg:key_operators}
    Let $\Omega\subseteq \R^d$ be open. Consider $\hilbert H_0 := L^2(\Omega)$, $\hilbert H_1 := L^2(\Omega)^d$, $\hilbert H_2 := L^2(\Omega)^{d\times d}$, and
    \begin{alignat*}{3}
       \op A_0 := \gradcirc,\qquad        \op A_1 := \Rotcirc.
    \end{alignat*}
    Then, $R(\op A_0) \subseteq N(\op A_1)$ is a consequence of the fact that $\Rotcirc(\gradcirc v) = \skw(D^2 v) = 0$ for all $v \in C_c^\infty(\Omega)$, where $D^2 v$ denotes the Hessian of $v$. Since $N(\op A_1)$ is closed in $L^2(\Omega)^d$, we obtain \eqref{eqn:range_ker_condition}.
\end{eg}

The main objective of this section is to identify two key sets, $\mathcal{D}$ and $V_\mathcal{D}$, and study their relation to the operators $\op A_0$ and $\op A_1$. We begin with the definition of these sets: 

\begin{dfntn}\label{defn:vd}
    For $\mathcal{D} \subseteq D(\op A_0^\ast)$, we define the collection
    \begin{align}\label{eqn:vd}
        V_{\mathcal{D}} = \left\{ \psi \in \hilbert H_1 \,:\, \psi = \left(\op I - \op A_0 (\op A_0^\ast \op A_0 + \op I)^{-1} \op A_0^\ast \right) \phi\text{ for some } \phi \in \mathcal{D} \right\}.
    \end{align}
\end{dfntn}

\begin{rmrk}
    Note that
    \begin{align}\label{eqn:vd_alt_helper}
        \op I - \op A_0 (\op A_0^\ast \op A_0 + \op I)^{-1} \op A_0^\ast \subseteq (\op A_0 \op A_0^\ast + \op I)^{-1},
    \end{align}
    which offers an alternative way of writing $V_\mathcal{D}$, namely,
    \begin{align*}
        V_\mathcal{D} =  \left\{ \psi \in \hilbert H_1 \,:\, \psi = (\op A_0 \op A_0^\ast + \op I)^{-1} \phi\text{ for some } \phi \in \mathcal{D} \right\} .
    \end{align*}
    To see \eqref{eqn:vd_alt_helper}, note that $\op A_0 (\op A_0^\ast \op A_0 + \op I)^{-1} \op A_0^\ast (\op A_0 \op A_0^\ast + \op I)
    = \op A_0 (\op A_0^\ast \op A_0 + \op I)^{-1} (\op A_0^\ast \op A_0 + \op I) \op A_0^\ast
    = \op A_0 \op A_0^\ast$, which gives $\op A_0 (\op A_0^\ast \op A_0 + \op I)^{-1} \op A_0^\ast \subseteq \op A_0 \op A_0^\ast (\op A_0^\ast \op A_0 + \op I)^{-1}$.
\end{rmrk}

The next result quotes a density lemma in its abstract form provided in \cite[Lemma 4.1]{PaulyWaurick2026}. For convenience of the reader, we provide the easy proof. 

\begin{prop}\label{prop:vd_common_core}
    If $\mathcal{D} \subseteq D(\op A_1) \cap D(\op A_0^\ast)$ and $\mathcal{D}$ is a core for $\op A_1$, then $V_\mathcal{D}$ is a common core of $\op A_1$ and $\op A_0^\ast$. That is, $V_\mathcal{D}$ is dense in $\left( D(\op A_1) \cap D(\op A_0^\ast), \langle \cdot, \cdot \rangle_{\hilbert X} \right)$ where
    \begin{align}\label{eqn:graph_inner_product}
        \langle u, v \rangle_{\hilbert X} := \langle u, v \rangle_{\hilbert H_1} 
        + \langle \op A_1 u, \op A_1 v \rangle_{\hilbert H_2} 
        + \langle \op A_0^\ast u, \op A_0^\ast v \rangle_{\hilbert H_0}
    \end{align}
    for all $u,v \in D(\op A_1) \cap D(\op A_0^\ast)$.
\end{prop}

\begin{proof}
    Since $\op A_1$ and $\op A_0^\ast$ are closed, $\left( D(\op A_1) \cap D(\op A_0^\ast), \langle \cdot, \cdot \rangle_{\hilbert X} \right)$ is a Hilbert space. Furthermore, $V_\mathcal{D}$ is a subspace of this Hilbert space, which follows by combining the assumption $\mathcal{D} \subseteq D(\op A_1) \cap D(\op A_0^\ast)$ with the definition of $V_{\mathcal{D}}$ (see \eqref{eqn:vd}).
    
    Hence, it suffices to show that $(V_\mathcal{D})^{\perp_{\hilbert X}}$ is trivial, where 
    \begin{align*}
        (V_\mathcal{D})^{\perp_{\hilbert X}} = \{ u \in D(\op A_1) \cap D(\op A_0^\ast) \,:\, \langle u, v \rangle_{\op X} = 0 ~~\text{for all $v \in V_\mathcal{D}$} \}.
    \end{align*}

    Fix $u \in (V_\mathcal{D})^{\perp_{\hilbert X}}$. We shall show that $u = 0$. To this end, let $\phi \in \mathcal{D}$. Then,
    \begin{alignat}{3}
        0 &= \langle u, \phi - \op A_0 (\op A_0^\ast \op A_0 + \op I)^{-1} \op A_0^\ast \phi \rangle_{\hilbert X}
        \qquad
        &&\text{(since $u \in (V_\mathcal{D})^\perp$)} \nonumber \\
        &= \langle u, \phi \rangle_{\hilbert H_1}
        - \langle u, \op A_0 (\op A_0^\ast \op A_0 + \op I)^{-1} \op A_0^\ast \phi \rangle_{\hilbert H_1} \nonumber\\
        &\quad+ \langle \op A_1 u, \op A_1 \phi \rangle_{\hilbert H_2}
        - \langle \op A_1 u, \op A_1 \op A_0 (\op A_0^\ast \op A_0 + \op I)^{-1} \op A_0^\ast \phi \rangle_{\hilbert H_2}
        \qquad
        &&\nonumber\\
        &\quad+ \langle \op A_0^\ast u, \op A_0^\ast \phi \rangle_{\hilbert H_0}
        - \langle \op A_0^\ast u, \op A_0^\ast \op A_0 (\op A_0^\ast \op A_0 + \op I)^{-1} \op A_0^\ast \phi \rangle_{\hilbert H_0}
        &&\text{(by \eqref{eqn:graph_inner_product}).} \label{eqn:graph_inner_product_simplified}
    \end{alignat}

    We observe that \eqref{eqn:graph_inner_product_simplified} can be further simplified: First, we note that the term $\langle \op A_1 u, \op A_1 \op A_0 (\op A_0^\ast \op A_0 + \op I)^{-1} \op A_0^\ast \phi \rangle_{\hilbert H_2}$ is zero, by \eqref{eqn:range_ker_condition}. Second, we note that the term $\langle \op A_0^\ast u, \op A_0^\ast \op A_0 (\op A_0^\ast \op A_0 + \op I)^{-1} \op A_0^\ast \phi \rangle_{\hilbert H_0}$ can be rewritten as
    \begin{align*}
        &\langle \op A_0^\ast u, \op A_0^\ast \op A_0 (\op A_0^\ast \op A_0 + \op I)^{-1} \op A_0^\ast \phi \rangle_{\hilbert H_0} \nonumber\\
        &\qquad= \langle \op A_0^\ast u, (\op A_0^\ast \op A_0 + \op I)(\op A_0^\ast \op A_0 + \op I)^{-1} \op A_0^\ast \phi \rangle_{\hilbert H_0}
        - \langle \op A_0^\ast u, (\op A_0^\ast \op A_0 + \op I)^{-1} \op A_0^\ast \phi \rangle_{\hilbert H_0} \nonumber\\
        &\qquad= \langle \op A_0^\ast u, \op A_0^\ast \phi \rangle_{\hilbert H_0}
        - \langle \op A_0^\ast u, (\op A_0^\ast \op A_0 + \op I)^{-1} \op A_0^\ast \phi \rangle_{\hilbert H_0} \nonumber\\
        &\qquad= \langle \op A_0^\ast u, \op A_0^\ast \phi \rangle_{\hilbert H_0}
        - \langle u, \op A_0 (\op A_0^\ast \op A_0 + \op I)^{-1} \op A_0^\ast \phi \rangle_{\hilbert H_1},
    \end{align*}
    where the final equality follows because $(\op A_0^\ast \op A_0 + \op I)^{-1} \op A_0^\ast \phi \in D(\op A_0)$.

    Continuing from \eqref{eqn:graph_inner_product_simplified}, we have
    \begin{align}\label{eqn:graph_inner_product_simplified_v2}
        0 = 
        \langle u, \phi \rangle_{\hilbert H_1}
        + \langle \op A_1 u, \op A_1 \phi \rangle_{\hilbert H_2}
        =: \langle u, \phi \rangle_{\op A_1}
        \qquad \text{for all} ~~ \phi \in \mathcal{D}.
    \end{align}
    Since $\op A_1$ is closed, $( D(\op A_1), \langle \cdot, \cdot \rangle_{\op A_1})$ is a Hilbert space. Since $\mathcal{D}$ is a core for $\op A_1$, $\mathcal{D}^{\perp_{\op A_1}} = \{ 0 \}$. We conclude from \eqref{eqn:graph_inner_product_simplified_v2} that $u = 0$. The proof is complete.
\end{proof}
We shall need the following abstract result to help us construct a core for $\op A_1^\ast \op A_1$.
\begin{lmm}\label{lem:core_criterion}
    Let $\op A : D(\op A) \subseteq \hilbert H \rightarrow \hilbert H$ be a self-adjoint non-negative operator on a Hilbert space $(\hilbert H, \langle \cdot, \cdot \rangle)$. Suppose that $D \subseteq D(\op A)$ is a dense subset of $\hilbert H$ such that $(\op I + \op A)[D]$ is also dense in $\hilbert H$, then $D$ is a core for $\op A$.
\end{lmm}

\begin{rmrk}
    The point of this lemma is that if one has a convenient characterization of the set $(\op I + \op A)[D]$, then it may be easier to check for the density of $(\op I + \op A)[D]$ in $\hilbert H$, instead of a direct check on the density of $D$ in $D(\op A)$ under the graph norm of $\op A$. See Lemma \ref{lem:sc_bijection} below.
\end{rmrk}

\begin{proof}
    [Proof of Lemma \ref{lem:core_criterion}]
    Consider the restriction $\op A|_D \subseteq \op A$. Then $\op A|_D$ is densely defined and closable, and we may consider the closed operator $\op A_D := \overline{\op A|_D}$. The assertion of the Lemma is equivalent to $\op A_D = \op A$.

    We first show that $\op I + \op A_D$ is surjective. Let $y \in \hilbert H$. By density of $(\op I + \op A)[D]$ in $\hilbert H$, there exists a sequence $(x_n)_n$ in $ D$ such that $y_n := (\op I + \op A)x_n \rightarrow y$ in $\hilbert H$ as $n\rightarrow \infty$. Consider the identity
    \begin{align*}
        \langle x_n, y_n \rangle = \langle x_n, x_n \rangle + \langle x_n, \op A x_n \rangle.
    \end{align*}
    By applying the Cauchy--Schwarz inequality on the left, and using the non-negativity assumption of $\op A$, we see that $(x_n)_n$ is bounded in $\hilbert H$. Choose a weakly convergent subsequence $(x_{n_k})_k$ with limit $x \in \hilbert H$. Then, $y_{n_k} = x_{n_k} + \op A_D x_{n_k}$ and $x_{n_k}$ converges weakly in $\hilbert H$, and thus $A_D x_{n_k}$ converges weakly in $\hilbert H$. Since $\op A_D$ is weakly closed, we conclude that $x \in D(\op A_D)$ and $\op A_D x = y - x$. So $\op I + \op A_D$ is surjective.

    We thus have $\op I + \op A_D \subseteq \op I + \op A$, where $\op I + \op A_D$ is surjective, and $\op I + \op A$ is injective (since $\op A \geq 0$). Thus, $\op A_D = \op A$ as there is no proper injective extension of a surjective map; see also \cite[Lemma 1.3]{Sch12}.
\end{proof}

We are equipped with the necessary tools for proving Gaffney's (in)equality for cubes and, thus, Theorem \ref{thrm:gaffcub} independently of regularity statements or more elaborate integration by parts formulas.

\section{Gaffney's equality on the unit cube}\label{sect:gaffney_cube}
Throughout this section, we set $\Omega \coloneqq (0,1)^d$ and assume that $d\geq 2$.
\begin{dfntn}\label{defn:sc}
    We define
    \begin{align*}
        \mathcal{S}_\mathcal{C} := 
        \Span{\{ \varphi_i^{(  k)}   e_i: i\in\{1,\ldots,d\},\;   k\in\N_0^d \}},
    \end{align*}
    where, for $i\in\{1,\ldots,d\}$ and $  k = (k_1,\ldots,k_d)^\top\in \N_0^d$,
    \begin{align*}
        \varphi_i^{(  k)}:\Omega\rightarrow \R,\qquad \varphi_i^{(  k)}(x_1,\ldots,x_d) := \cos(\pi k_i x_i) \prod_{\substack{1\leq j \leq d \\ j\neq i}} \sin(\pi k_j x_j),
    \end{align*}
    and $  e_i$ denotes the $i$-th column of the identity matrix in $\R^{d\times d}$.
\end{dfntn}
\begin{lmm}\label{lem:sc_in_domain_rotcirc}
    $\mathcal{S}_\mathcal{C} \subseteq D(\Rotcirc)$.
\end{lmm}

\begin{proof}
    Let $1\leq i\leq d$ and $  k = (k_1,\ldots,k_d)^\top\in \N_0^d$ be fixed, and set $  v:= \varphi_i^{(  k)}   e_i$. Note that $  v\in C^\infty(\overline{\Omega})^d\subseteq D(\Rot)$. We need to show that there exists a sequence $(  v_n)_n$ in $ C_c^\infty(\Omega)^d$ such that 
    \begin{align}\label{Sc approx dom}
    \|  v_n -   v\|_{H(\Rot;\Omega)}^2 = \|  v_n -   v\|_{L^2(\Omega)^d}^2 + \|\Rot (  v_n -   v)\|_{L^2(\Omega)^{d\times d}}^2  \xrightarrow[n\to\infty]{} 0.
    \end{align}
    To this end, note $\cos(\pi k_i\,\cdot)\in L^2(T)$ and $\sin(\pi k_j\,\cdot)\in H^1_0(T)$ for all $1\leq j\leq d$, where $T:=(0,1)$. Hence, there exist sequences $(\psi_n^{(1)})_n,\ldots,(\psi_n^{(d)})_n$ in $ C_c^\infty(T)$ such that for $j\neq i$
    \begin{align}\label{sincos}
        \lim_{n\rightarrow \infty} \|\cos(\pi k_i\,\cdot) - \psi_n^{(i)}\|_{L^2(T)} = 0,\quad  \lim_{n\rightarrow \infty} \|\sin(\pi k_j\,\cdot) - \psi_n^{(j)}\|_{H^1(T)} = 0.
    \end{align}
    For $n\in \N$, we set $  v_n := \varphi_{i,n}^{(  k)}   e_i \in C_c^\infty(\Omega)^d$, where 
    \begin{align*}
        \varphi_{i,n}^{(  k)}:\Omega\rightarrow\R,\qquad \varphi_{i,n}^{(  k)}  (x_1,\ldots,x_d) := \psi_n^{(i)}(x_i)\prod_{\substack{1\leq j \leq d \\ j\neq i}} \psi_n^{(j)}(x_j).
    \end{align*}
    Then, in view of \eqref{sincos}, we have that 
    \begin{align*}
        \|  v_n -   v\|_{L^2(\Omega)^d} = \|(\varphi_{i,n}^{(  k)}- \varphi_{i}^{(  k)})  e_i\|_{L^2(\Omega)^d} = \|\varphi_{i,n}^{(  k)}- \varphi_{i}^{(  k)}\|_{L^2(\Omega)} \xrightarrow[n\to\infty]{} 0
    \end{align*}
    and
    \begin{align*}
        \|\Rot (  v_n -   v)\|_{L^2(\Omega)^{d\times d}} &= \frac{1}{\sqrt{2}}\left\|\sum_{j = 1}^d \partial_j (\varphi_{i,n}^{(  k)}- \varphi_{i}^{(  k)}) (  e_i   e_j^\top -   e_j   e_i^\top)\right\|_{L^2(\Omega)^{d\times d}} \\ &\leq \sum_{\substack{j = 1 \\ j\neq i}}^d \|\partial_j (\varphi_{i,n}^{(  k)}- \varphi_{i}^{(  k)})\|_{L^2(\Omega)} \xrightarrow[n\to\infty]{} 0,
    \end{align*}
    i.e., \eqref{Sc approx dom} holds. We conclude that $  v\in \overline{C_c^\infty(\Omega)^d}^{\| \cdot \|_{H(\Rot;\Omega)}} = D(\Rotcirc)$.
\end{proof}

\begin{cor}\label{cor:sc_in_domain_rotcircstarrotcirc_and_div}
    We have $\mathcal{S}_\mathcal{C} \subseteq D((\Rotcirc)^\ast (\Rotcirc))$ and $\mathcal{S}_\mathcal{C} \subseteq D(\Rotcirc) \cap D(\div)$.
\end{cor}

\begin{proof}
    By Lemma \ref{lem:sc_in_domain_rotcirc} and $D((\Rotcirc)^\ast) = D(-\Div \skw)$, the proof of the first containment reduces to checking that $-\Div \skw{ \Rotcirc   u} \in L^2(\Omega)^d$, whenever $  u \in \mathcal{S}_\mathcal{C}$, which is immediate because $  u\in C^\infty(\overline{\Omega})^d$.
    
    In a similar fashion, we have $\div\,   u \in L^2(\Omega)$, which gives $  u \in H(\div;\Omega) = D(\div)$, and hence the second inclusion .
\end{proof}

\begin{lmm}\label{lem:firstformula} Let $1\leq i \leq d$ and $  k = (k_1,\ldots,k_d)^\top \in \N_0^d$, $  u = \varphi_i^{(  k)}   e_i$. Then, $  u \in D((\Rotcirc)^\ast (\Rotcirc))$ and
\begin{align}\label{eqn:sc_calculation_step4}
    \begin{split}
        (\Rotcirc)^\ast (\Rotcirc)   u = \pi^2\sum_{j = 1}^d \left(k_j^2 \varphi_i   e_i - k_i k_j \varphi_j   e_j\right) = \pi^2\lvert   k\rvert^2\varphi_i   e_i - \pi^2 k_i \sum_{j = 1}^d k_j \varphi_j   e_j.
        \end{split}
    \end{align}    
\end{lmm}
\begin{proof}
 Since $  u \in D((\Rotcirc)^\ast (\Rotcirc))$ by Corollary \ref{cor:sc_in_domain_rotcircstarrotcirc_and_div}, $(\Rotcirc)^\ast (\Rotcirc)   u$ is well-defined. To keep the notation compact, we shall suppress the dependence on $  k$, and rather write, for instance, $  u = \varphi_i   e_i$. First, noting that $\Grad  u =   e_i (\nabla \varphi_i)^\top$, we have that
    \begin{align*}
        \sqrt{2}\,\Rotcirc   u =   e_i (\nabla \varphi_i)^\top - (\nabla \varphi_i)   e_i^\top = \sum_{j = 1}^d \partial_j \varphi_i (  e_i   e_j^\top -   e_j   e_i^\top).
    \end{align*}
    Upon recalling that $(\Rotcirc)^\ast = -\Div \skw$, it follows that
    \begin{align}
        (\Rotcirc)^\ast (\Rotcirc)   u
         = - \Div(\sqrt{2}\,\Rotcirc   u) = \sum_{j = 1}^d ( (\partial_{ij}^2 \varphi_i)   e_j - (\partial_j^2 \varphi_i)   e_i ). \label{eqn:sc_calculation_step2}
    \end{align}
    We can further rewrite the terms in sum \eqref{eqn:sc_calculation_step2} by using that
    \begin{align*}
        \partial_{ij}^2 \varphi_i = -\pi^2 k_i k_j \varphi_j
    \end{align*}
    for any $1\leq j\leq d$. In particular, \eqref{eqn:sc_calculation_step2} becomes
    \begin{align*}
    \begin{split}
        (\Rotcirc)^\ast (\Rotcirc)   u = \pi^2\sum_{j = 1}^d \left(k_j^2 \varphi_i   e_i - k_i k_j \varphi_j   e_j\right) = \pi^2\lvert   k\rvert^2\varphi_i   e_i - \pi^2 k_i \sum_{j = 1}^d k_j \varphi_j   e_j,
        \end{split}
    \end{align*}
which is the assertion.
\end{proof}

\begin{prop}\label{prop:sc_vs_eigenspaces}
    We have
    \begin{align}
        \mathcal{S}_\mathcal{C} 
        &\subseteq \Span{ \left\{ \bigcup_{\lambda \leq 0} N\left(\lambda + (\Rotcirc)^\ast (\Rotcirc) \right) \right\}}. \label{eqn:sc_includes}
    \end{align}
\end{prop}

\begin{proof}
    Fix $1\leq i \leq d$ and $  k = (k_1,\ldots,k_d)^\top \in \N_0^d$. By linearity, it suffices to check that $  u = \varphi_i^{(  k)}   e_i$ can be expressed as a finite linear combination of eigenfunctions of $(\Rotcirc)^\ast (\Rotcirc)$. Again, we shall suppress the dependence on $  k$.

    $  k \neq   0$:  We shall use Lemma \ref{lem:firstformula}, \eqref{eqn:sc_calculation_step4}, to express $(\Rotcirc)^\ast (\Rotcirc)   u$ as a finite linear combination of eigenfunctions of $(\Rotcirc)^\ast (\Rotcirc)$. We write
    \begin{align*}
          u = \frac{k_i}{\lvert   k\rvert^2}  u_1 +   u_2,\qquad\text{where}\qquad   u_1 := \sum_{n=1}^d k_n \varphi_n   e_n,\quad \text{and} \quad   u_2 :=   u - \frac{k_i}{\lvert   k\rvert^2}  u_1,
    \end{align*}
    and claim that $  u_1\in N\left((\Rotcirc)^\ast (\Rotcirc)\right)$ and $  u_2\in N\left((\Rotcirc)^\ast (\Rotcirc) - \pi^2 |  k|^2 \right)$.

    Indeed, note that $  u_1\in D\left((\Rotcirc)^\ast (\Rotcirc)\right)$ and $  u_2\in D\left((\Rotcirc)^\ast (\Rotcirc)- \pi^2 |  k|^2\right)=D\left((\Rotcirc)^\ast (\Rotcirc)\right)$ by Corollary \ref{cor:sc_in_domain_rotcircstarrotcirc_and_div}, and that in view of \eqref{eqn:sc_calculation_step4} there holds
    \begin{align*}
        (\Rotcirc)^\ast (\Rotcirc)  u_1 &= \sum_{n=1}^d k_n \left[\pi^2\lvert   k\rvert^2 \varphi_n   e_n - \pi^2 k_n \sum_{j = 1}^d k_j \varphi_j   e_j \right] \\
        &= \pi^2 \sum_{n=1}^d \sum_{j = 1}^d k_j^2 k_n \varphi_n   e_n - \pi^2 \sum_{n=1}^d \sum_{j = 1}^d k_n^2 k_j \varphi_j   e_j =   0,
    \end{align*}
    and 
    \begin{align*}
        (\Rotcirc)^\ast (\Rotcirc)   u_2 = (\Rotcirc)^\ast (\Rotcirc)   u = \pi^2 \lvert   k\rvert^2 \varphi_i   e_i - \pi^2 k_i \sum_{j=1}^d k_j \varphi_j   e_j = \pi^2 \lvert   k\rvert^2   u_2.
    \end{align*}
    Altogether, $u \in N\left((\Rotcirc)^\ast (\Rotcirc)\right) + N\left((\Rotcirc)^\ast (\Rotcirc)-\pi^2 \lvert   k\rvert^2\right)$.    

    $k = 0$: This gives $\varphi_i^{(  k)} = \varphi_i^{(  0)} \equiv 0$, which implies that $u \equiv 0$ as $d \geq 2$, and so the inclusion \eqref{eqn:sc_includes} is trivially satisfied. 
    This concludes the proof.
\end{proof}
We now strengthen Lemma \ref{lem:sc_in_domain_rotcirc} as follows.
\begin{prop}\label{prop:sc_core_rotcirc}
    $\mathcal{S}_\mathcal{C}$ is a core for $\Rotcirc$.
\end{prop}
To prove Proposition \ref{prop:sc_core_rotcirc}, we need a preparatory lemma:
\begin{lmm}\label{lem:sc_bijection}
    $\op I + (\Rotcirc)^{\ast}(\Rotcirc) : \mathcal{S}_\mathcal{C} \rightarrow \mathcal{S}_\mathcal{C}$ is a bijection.
\end{lmm}
\begin{proof}
    First, we note that $(\Rotcirc)^\ast (\Rotcirc)$ is non-negative, and thus $\op I + (\Rotcirc)^\ast (\Rotcirc) : D((\Rotcirc)^\ast (\Rotcirc)) \rightarrow L^2(\Omega)^d$ is a bijection. Moreover, the computations in Proposition \ref{prop:sc_vs_eigenspaces} (see \eqref{eqn:sc_calculation_step4}) show that $(\Rotcirc)^\ast (\Rotcirc)$ maps $\mathcal{S}_\mathcal{C}$ to $\mathcal{S}_\mathcal{C}$. Thus, with a slight abuse of notation, the map $\op I + (\Rotcirc)^\ast (\Rotcirc) : \mathcal{S}_\mathcal{C} \rightarrow \mathcal{S}_\mathcal{C}$ is well-defined and injective.
    
    It remains to show that $\op I + (\Rotcirc)^\ast (\Rotcirc)$ maps $\mathcal{S}_\mathcal{C}$ onto $\mathcal{S}_\mathcal{C}$. For each $n \in \N_0$, consider the finite-dimensional subspace $\mathcal{S}_\mathcal{C}(n)$ of $\mathcal{S}_\mathcal{C}$ by restricting the frequency $  k$ component-wise to a maximum of $n$:
    \begin{align*}
        \mathcal{S}_\mathcal{C}(n) := \Span{\{ \varphi_i^{(  k)}   e_i: i\in\{1,\ldots,d\},\;   k\in\N_0^d \cap [0,n]^d \}}.
    \end{align*}
    Then, trivially, $\mathcal{S}_\mathcal{C}= \bigcup_{n\in \N_0} \mathcal{S}_\mathcal{C}(n)$, and \eqref{eqn:sc_calculation_step4} shows that $\op I + (\Rotcirc)^\ast (\Rotcirc) : \mathcal{S}_\mathcal{C}(n) \rightarrow \mathcal{S}_\mathcal{C}(n)$ is well-defined and injective, and hence also bijective by the rank-nullity theorem on $\mathcal{S}_\mathcal{C}(n)$. The proof is complete.
\end{proof}
\begin{proof}
    [Proof of Proposition \ref{prop:sc_core_rotcirc}]
    By \cite[Prop. 3.18]{Sch12}, it suffices to show that $\mathcal{S}_\mathcal{C}$ is a core for $(\Rotcirc)^\ast(\Rotcirc)$. We shall show this by applying Lemma \ref{lem:core_criterion}, with $\op A = (\Rotcirc)^\ast (\Rotcirc)$ and $D = \mathcal{S}_\mathcal{C}$. 
    
    Since $\{ \sin{(k\pi \cdot)} : k \in \N \}$ and $\{ \cos{(k\pi \cdot)} : k \in \N_0 \}$ are complete orthogonal systems in $L^2(0,1)$, the set $\mathcal{S}_\mathcal{C}$ is dense in $L^2(\Omega)^d$. Moreover, Lemma \ref{lem:sc_bijection} gives $(\op I + (\Rotcirc)^\ast (\Rotcirc))[\mathcal{S}_\mathcal{C}] = \mathcal{S}_\mathcal{C}$. It follows that the hypotheses of Lemma \ref{lem:core_criterion} are satisfied, and the proof is complete.
\end{proof}
We also need the mapping properties of $(\op I + (\gradcirc)(\gradcirc)^\ast)^{-1}$ under $\mathcal{S}_\mathcal{C}$.
Below we prove a property that explicitly uses the trigonometric structure of $\mathcal{S}_\mathcal{C}$. Recalling $V_\mathcal{D}$ from \eqref{eqn:vd} in the special set-up here coming from Example \ref{eg:key_operators}, we have the following seemingly surprising fact.
 \begin{lmm}\label{lem:vsc_in_sc}
    $V_{\mathcal{S}_\mathcal{C}} \subseteq \mathcal{S}_\mathcal{C}$. 
\end{lmm}

\begin{proof}
    Fix $1 \leq i \leq d$ and $  k \in \N_0^d$ and consider the function $  u = \varphi_i^{(  k)}   e_i$. Set $\op A_0 := \gradcirc$. By \eqref{eqn:vd}, we would like to show that 
    \begin{align}\label{eqn:vsc_in_sc_claim}
        \op A_0 (\op A_0^\ast \op A_0 + \op I)^{-1} \op A_0^\ast   u \in \mathcal{S}_\mathcal{C}.
    \end{align}
    This follows from the fact that $\op A_0^\ast \op A_0 = -\div \gradcirc$ is the negative Dirichlet Laplacian on $\Omega$, whose eigenfunctions are of the form $\psi^{(\tilde{  k})} (  x) = \Pi_{j=1}^d \sin{(\pi \tilde{k}_j x_j)}$, $\tilde{  k} \in \N^d$.

    Indeed, we first see that $\op A_0^\ast   u = -\div\,   u = - \partial_i \varphi_i^{(  k)} = \pi k_i \psi^{(  k)}$. If $  k = (k_j)_j$ is such that some $k_j = 0$, then $\op A_0^\ast   u = 0$ and \eqref{eqn:vsc_in_sc_claim} is immediate. Thus, we may assume that $  k \in \N^d$. That is, $\op A_0^\ast   u$ is an eigenfunction of $\op A_0^\ast \op A_0$. This implies that $v = (\op A_0^\ast \op A_0 + \op I)^{-1} \op A_0^\ast   u$ is a constant multiple of $\psi^{(  k)}$. Combining this with the identity
    \begin{align*}
        \op A_0 \psi^{(  k)}
        = \gradcirc \psi^{(  k)}
        = \sum_{j=1}^d (\pi k_j) \varphi_j^{(  k)}   e_j,
    \end{align*}
    we conclude that $\op A_0 v \in \mathcal{S}_\mathcal{C}$. This gives \eqref{eqn:vsc_in_sc_claim}, as required. 
\end{proof}

By the abstract framework of Section \ref{sect:abstract_framework} and the properties of $\mathcal{S}_\mathcal{C}$ established in this section, we may now upgrade Proposition \ref{prop:sc_core_rotcirc} to obtain the following result:

\begin{cor}\label{cor:sc_core_rotcirc_and_gradcric}
    $\mathcal{S}_\mathcal{C}$ is a common core for $\Rotcirc$ and $-\div$.
\end{cor}

\begin{proof}
    Set $\op A_0 := \gradcirc$, $\op A_1 := \Rotcirc$, and $\hilbert H_1 := L^2(\Omega)^d$. We know from Corollary \ref{cor:sc_in_domain_rotcircstarrotcirc_and_div} and Proposition \ref{prop:sc_core_rotcirc} that $\mathcal{S}_{\mathcal{C}}\subseteq D(\Rotcirc)\cap D(\div) = D(\op A_1)\cap D(\op A_0^\ast)$ and that $\mathcal{S}_{\mathcal{C}}$ is a core for $\Rotcirc$. Hence, by Proposition \ref{prop:vd_common_core}, the set $V_{\mathcal{S}_{\mathcal{C}}}$ defined in \eqref{eqn:vd} with $\mathcal{D} = \mathcal{S}_{\mathcal{C}}$ is a common core for $\op A_1 = \Rotcirc$ and $\op A_0^\ast = -\div$. Since $V_{\mathcal{S}_\mathcal{C}} \subseteq \mathcal{S}_\mathcal{C}$ (Lemma \ref{lem:vsc_in_sc}), the conclusion follows.
\end{proof}

We are now in a position to prove Gaffney's equality, the main result of the section.

\begin{thrm}\label{thm:gaffney_cube}
    Let $\Omega :=(0,1)^d$ and $  u\in H_0(\Rot;\Omega)\cap H(\div;\Omega)$. Then, there holds $  u\in H^1(\Omega)^d$ and
    \begin{align}\label{Gaff}
        \|\Grad   u\|_{L^2(\Omega)^{d\times d}}^2 = \|\Rotcirc   u\|_{L^2(\Omega)^{d\times d}}^2 + \|\div\,   u\|_{L^2(\Omega)}^2.
    \end{align}
\end{thrm}

\begin{proof}
    \textbf{Step 1.}    We show that the equality \eqref{Gaff} holds for all functions in $\mathcal{S}_{\mathcal{C}}$. 
    
    To this end, let $  v\in \mathcal{S}_{\mathcal{C}}$ be fixed. We begin by observing that
    \begin{align*}
        &\|\Grad  v\|_{L^2(\Omega)^{d\times d}}^2 \\
        &\quad= \frac{1}{2}\|\Grad  v- (\Grad  v)^\top\|_{L^2(\Omega)^{d\times d}}^2 + (\Grad  v,(\Grad  v)^\top)_{L^2(\Omega)^{d\times d}}  \\
        &\quad= \|\Rotcirc   v\|_{L^2(\Omega)^{d\times d}}^2 + (\Grad  v,(\Grad  v)^\top)_{L^2(\Omega)^{d\times d}} ,
    \end{align*}
    and it only remains to show that $\|\div\,   v\|_{L^2(\Omega)}^2 = (\Grad  v,(\Grad  v)^\top)_{L^2(\Omega)^{d\times d}}$. Integrating by parts twice yields
    \begin{align*}
        (\Grad  v, (\Grad  v)^\top)_{L^2(\Omega)^{d\times d}} &= \sum_{i,j = 1}^d (\partial_j v_i,\partial_i v_j)_{L^2(\Omega)}\\ &= \sum_{i,j = 1}^d (\partial_i v_i,\partial_{j} v_j)_{L^2(\Omega)} + b(  v) \\ &= \|\div\,  v\|_{L^2(\Omega)}^2 + b(  v),
    \end{align*}
    where the term $b(  v)$ is given by
    \begin{align*}
        b(  v):= \sum_{i,j = 1}^d\int_{\partial\Omega}\left( v_i (\partial_i v_j) n_j - v_i (\partial_j v_j) n_i\right) = \int_{\partial\Omega} \left(  v\cdot \partial_{  n}   v - (  v\cdot   n)\div\,   v\right)
    \end{align*}
    with $\partial_{  n}   v := (\Grad  v)^\top   n$. Let $F\subseteq \partial\Omega$ be a boundary face of $\Omega$. Since $ v$ has vanishing tangential trace on $\partial\Omega$, we have $\left.  v\right\rvert_F =   v_{F}^\perp   n_F$ with $  v_{F}^\perp := \left.   v\right\rvert_F\cdot   n_F$. Moreover, noting that $  n_F := \left.  n\right\rvert_F$ is constant, we obtain
    \begin{align*}
        \left.\left(  v\cdot \partial_{  n}   v\right)\right\rvert_F =   v_{F}^\perp   n_F\cdot(\Grad  (  v_{F}^\perp   n_F)^\top  n_F) &=   v_{F}^\perp   n_F\cdot \grad   v_{F}^\perp \\ &=   v_{F}^\perp \,\div (  v_{F}^\perp   n_F) = \left.\left((  v\cdot   n)\div\,   v\right)\right\rvert_F.
    \end{align*}
    Since $F$ was an arbitrary boundary face of $\Omega$, we deduce that $b(  v) = 0$.

    \textbf{Step 2.} Let $  u\in H_0(\Rot;\Omega)\cap H(\div;\Omega) = D(\Rotcirc)\cap D(\div)$. By Corollary \ref{cor:sc_core_rotcirc_and_gradcric}, there exists a sequence $(  u_k)_k$ in $ \mathcal{S}_{\mathcal{C}}$ such that 
    \begin{align*}
        \lim_{k\rightarrow \infty} \|  u_k -   u\|_{L^2(\Omega)^d} =\lim_{k\rightarrow \infty}\|\div (  u_k -   u)\|_{L^2(\Omega)} = \lim_{k\rightarrow \infty}\|\Rotcirc (  u_k -   u)\|_{L^2(\Omega)^{d\times d}} = 0.
    \end{align*}
    By Step 1, for all $k\in \N$, we have 
    \begin{align}\label{Gaff k}
        \|\Grad   u_k\|_{L^2(\Omega)^{d\times d}}^2 = \|\Rotcirc   u_k\|_{L^2(\Omega)^{d\times d}}^2 + \|\div\,   u_k\|_{L^2(\Omega)}^2.
    \end{align}
    In particular, $(\Grad  u_k)_k$ is a Cauchy sequence in $L^2(\Omega)^{d\times d}$. We deduce that $  u\in D(\Grad) =  H^1(\Omega)^d$ and $\|  u_k -   u\|_{H^1(\Omega)^d}\rightarrow 0$. Passing to the limit in \eqref{Gaff k}, we find \eqref{Gaff}. 
\end{proof}

Lest the reader think that it is possible to generalize Theorem \ref{thm:gaffney_cube} to arbitrary Lipschitz domains, we point out the following counterexample to Gaffney's (in)equality:

\begin{rmrk}
    Consider the two-dimensional L-shaped domain $\Omega = (0,2)^2 \setminus [1,2]^2$. Fix a right-hand side $f \in L^2(\Omega)$ for which the weak solution $u := (-\div \gradcirc)^{-1} f \in H^1_0(\Omega)$ to the Dirichlet problem
    \begin{align*}
    \begin{cases}
        -\Delta u = f& \text{in }\Omega,\\
        \hfill u = 0& \text{on }\partial\Omega
    \end{cases}
    \end{align*}
    satisfies $u\notin H^2(\Omega)$. Then, $  v := \gradcirc u\notin H^1(\Omega)^2$, but
    \begin{itemize}
        \item $ v \in H(\div;\Omega)$ with $-\div\,   v = f$, and

        \item $  v \in H_0(\Rot;\Omega)$ with $\Rotcirc   v = 0$ since
        \begin{align*}
            \langle   v, (\Rotcirc)^\ast A\rangle_{L^2(\Omega)^2} = \langle \gradcirc u, -\Div (\skw A)\rangle_{L^2(\Omega)^2} = \langle u, D^2:(\skw A)\rangle_{L^2(\Omega)} = 0
        \end{align*}
        for any $A\in C^\infty(\Omega)^{2\times 2}$.
    \end{itemize}
    In particular, \eqref{Gaff} does not hold, even after replacing the equality by `$\lesssim$'.
\end{rmrk}

Let us briefly comment on how to obtain an analogue of Theorem \ref{thm:gaffney_cube} for the space $H(\Rot;\Omega)\cap H_0(\div;\Omega)$.
\begin{rmrk}\label{rmk:gaffney_with_divzero_cs}
   Let $\Omega :=(0,1)^d$ and $\mathcal{C}_{\mathcal{S}} :=
    \Span{\{ \psi_i^{(  k)}   e_i: i\in\{1,\ldots,d\},\;   k\in\N_0^d \}}$, where 
\begin{align*}
    \psi_i^{(  k)}:\Omega\rightarrow \R,\qquad \psi_i^{(  k)}(x_1,\ldots,x_d) := \sin(\pi k_i x_i) \prod_{\substack{1\leq j \leq d \\ j\neq i}} \cos(\pi k_j x_j)
\end{align*}
for $i\in\{1,\ldots,d\}$ and $  k = (k_1,\ldots,k_d)^\top\in \N_0^d$. Similar arguments to those presented in this section, but with $A_0 = \grad$ and $A_1 = \Rot$, lead to $\mathcal{C}_{\mathcal{S}}$ being a common core for $\Rot$ and $-\divcirc$, and to the following analogue of Theorem \ref{thm:gaffney_cube}: there holds $H(\Rot;\Omega)\cap H_0(\div;\Omega) \subseteq H^1(\Omega)^d$ and
    \begin{align*}
        \|\Grad   u\|_{L^2(\Omega)^{d\times d}}^2 = \|\Rot   u\|_{L^2(\Omega)^{d\times d}}^2 +  \|\divcirc   u\|_{L^2(\Omega)}^2
    \end{align*}
for all $u\in H(\Rot;\Omega)\cap H_0(\div;\Omega)$.
\end{rmrk}

Finally, we shall fill the last gap in the reasoning of the proof of Picard--Weber--Weck's selection theorem by showing the chain rule formulas for $\div$ and $\Rotcirc$ in the next section. Similar formulas can be established for $\divcirc$ and $\Rot$, using  Remark \ref{rmk:gaffney_with_divzero_cs}, to also obtain the compactness statement with the other set of boundary conditions.

\section{\texorpdfstring{Proof of Proposition \ref{prop:rot_div_comp}}{Proof of compatibility formulae}}\label{sec:compatibility_proof}

We begin by fixing notation for this section. Let $\Omega \subseteq \R^d$ be a global Lipschitz domain with admissible map $\Phi:Q\rightarrow \Omega$, $Q = (0,1)^d$. Recall from Definition \ref{defn:global_lipschitz} that this implies $\Phi \in C^{0,1}_{b}(\overline{Q}, \overline{\Omega})$, $\Phi^{-1} \in C^{0,1}_{b}(\overline{\Omega}, \overline{Q})$, and there exists $c>0$ such that
\begin{align*}
    \det{J_\Phi(x)} \geq c \quad \text{for $a.e.$ $x \in Q$}.
\end{align*}
To differentiate between functions on $\Omega$ and $Q$, we introduce the following abbreviation.

\begin{notation}
    For $u:\Omega \rightarrow \R^n$, $n \in \N$, set $\widetilde{u} := u \circ \Phi$.
\end{notation}
\begin{rmrk}
    We record some preliminary facts that will be used throughout this section.
    \begin{itemize}
        \item By Rademacher's theorem, $\Phi$ is differentiable almost everywhere in $Q$ and $\Phi \in W^{1,\infty}(Q)^d$.
    
        \item Since $\Phi$ is a homeomorphism, it takes interior points to interior points, and boundary points to boundary points, hence $u \in C_c^{0,1}(\Omega)$ implies $\widetilde{u} \in C_c^{0,1}(Q)$.
    
        \item If $  u \in L^2(\Omega)^d$, then $\widetilde{u} \in L^2(Q)^d$ by a change of variables and $\det{J_\Phi} \in L^\infty(Q)$.
    
        \item We have that $\adjugate J_\Phi = (\det{J_\Phi}) J_\Phi^{-1} \in L^\infty(Q)^{d\times d}$.
    \end{itemize}
\end{rmrk}

\begin{prop}\label{prop:grad_composition}
    The following assertions are true.
    \begin{itemize}
        \item[(i)] If $u \in H^1(\Omega)$, then $\widetilde{u} \in H^1(Q)$ and $\grad_Q \widetilde{u} = J_\Phi^\top \widetilde{\grad_\Omega u}$.
        \item[(ii)] If $u \in H_0^1(\Omega)$, then $\widetilde{u} \in H_0^1(Q)$ and $\gradcirc_Q \widetilde{u} = J_\Phi^\top \widetilde{\gradcirc_\Omega u}$.
        \item[(iii)] If $  u\in H^1(\Omega)^d$, then $\widetilde{  u} \in H^1(Q)^d$ and $\Grad_Q \widetilde{  u} =  \widetilde{\Grad_\Omega   u}\, J_\Phi$.
        \item[(iv)] If $  u\in H^1_0(\Omega)^d$, then $\widetilde{  u} \in H_0^1(Q)^d$ and $\Gradcirc_Q \widetilde{  u} =  \widetilde{\Gradcirc_\Omega   u}\, J_\Phi$.
    \end{itemize}
\end{prop}

\begin{proof}
    See \cite[Theorem A.1]{PaulyWaurick2026}.
\end{proof}

\begin{prop}\label{prop:rot_composition}
    Let $  u \in H_0(\Rot;\Omega)$. Then, 
    $J_\Phi^\top \widetilde{  u} \in H_0(\Rot;Q)$ and
    \begin{equation*}
        \Rotcirc_Q{(J_\Phi^\top \widetilde{  u})} 
        =  J_\Phi^\top\,\widetilde{\Rotcirc_\Omega   u}\,J_\Phi =: ( J_\Phi^\top \otimes J_\Phi^\top )\widetilde{\Rotcirc_\Omega   u}.
    \end{equation*}
\end{prop}
\begin{proof}
    Set $J:=J_\Phi =  \Grad(\Phi)$. \textbf{Step 1.} We show that for any $  v\in C^{0,1}_c(\Omega)^d$ we have that $J^\top\widetilde{  v}\in H_0(\Rot;Q)$ and $\Rotcirc{(J^\top \widetilde{  v})} = ( J^\top \otimes J^\top )\widetilde{\Rotcirc   v}$.
    
    To this end, let $  v\in C^{0,1}_c(\Omega)^d$. For any $A\in C^\infty(Q)^{d\times d}$, we have that
    \begin{align*}
        \langle J^\top  \widetilde{  v}, (\Rotcirc_Q)^\ast A\rangle_{L^2(Q)^d} 
        &= \langle J^\top  \widetilde{  v}, -\Div_Q (\skw A)\rangle_{L^2(Q)^d}  \\
        &= -\langle\widetilde{  v}, J \Div_Q (\skw A)\rangle_{L^2(Q)^d} \\
        &= \langle\Gradcirc_Q(\widetilde{  v}), J \skw A\rangle_{L^2(Q)^{d\times d}} \\
        &= \langle J^\top \Gradcirc_Q(\widetilde{  v}), \skw A\rangle_{L^2(Q)^{d\times d}} \\
        &= \langle J^\top\widetilde{\Gradcirc_\Omega(  v)} J,\skw A\rangle_{L^2(Q)^{d\times d}} \\
        &= \langle\skw(J^\top\widetilde{\Gradcirc_\Omega(  v)} J), A\rangle_{L^2(Q)^{d\times d}}
        = \langle J^\top\,\widetilde{\Rotcirc_\Omega   v}\, J,A\rangle_{L^2(Q)^{d\times d}},
    \end{align*}
    where we have used in the third equality the following integration by parts identity:
    \begin{align*}
        \int_Q \Gradcirc(\widetilde{  v}) : \left(  \Grad(\Phi)\, B\right) = -\int_Q\widetilde{  v}\cdot \left(\Grad(\Phi)\, \Div (B)\right),
    \end{align*}
    for all $B\in C^\infty(Q;\R^{d\times d}_{\mathrm{skw}})$,
    where $\R^{d\times d}_{\mathrm{skw}}:= \{M\in \R^{d\times d}:M^\top = -M\}$. Since $J^\top\,\widetilde{\Rotcirc_\Omega   v}\, J\in L^2(Q)^{d\times d}$, we deduce from the above calculation that $J^\top  \widetilde{  v}\in D(\Rotcirc_Q) = H_0(\Rot;Q)$ and
    \begin{align*}
        \Rotcirc(J^\top  \widetilde{  v}) = J^\top\,\widetilde{\Rotcirc   v}\, J = (J^\top \otimes J^\top)\widetilde{\Rotcirc   v}.
    \end{align*}

    \textbf{Step 2.} Let $  u\in H_0(\Rot;\Omega
    )$. There exists a sequence $(  u_n)_n\subseteq C_c^\infty(\Omega)^d$ such that $\|  u -   u_n\|_{H(\Rot;\Omega)} \rightarrow 0$ as $n\rightarrow\infty$. First, we observe that this implies that $\|\widetilde{  u} - \widetilde{  u}_n\|_{L^2(Q)^d}\rightarrow 0$ and $\|\widetilde{\Rotcirc_\Omega   u} - \widetilde{\Rotcirc_\Omega   u_n}\|_{L^2(Q)^{d\times d}}\rightarrow 0$ as $n\rightarrow\infty$. Set 
    \begin{align*}
          w_n := J^\top \widetilde{  u}_n,
    \end{align*}
    and note that $  w_n \in H_0(\Rot;Q)$ by Step 1 with $\Rotcirc(  w_n) = (J^\top \otimes J^\top)\widetilde{\Rotcirc   u_n}$. Then, $  w_n \rightarrow J^\top \widetilde{  u}$ in $L^2(Q)^d$ and 
    \begin{align*}
        \Rotcirc_Q(  w_n)= (J^\top \otimes J^\top)\widetilde{\Rotcirc_\Omega   u_n}\rightarrow  (J^\top \otimes J^\top)\widetilde{\Rotcirc_\Omega   u} \quad\text{in }L^2(Q)^{d\times d}
    \end{align*}
    as $n\rightarrow\infty$. Noting that $(  w_n)_n$ is Cauchy in $H(\Rot;Q)$ and that $H_0(\Rot;Q)$ is closed in $H(\Rot;Q)$, we obtain that $J^\top \widetilde{  u} \in H_0(\Rot;Q)$ and $\Rotcirc_Q(J^\top \widetilde{  u}) = (J^\top \otimes J^\top)\widetilde{\Rotcirc_\Omega   u}$.
\end{proof}

\begin{prop}\label{prop:div_composition}
    Let $  u \in H(\div;\Omega)$. Then $(\adjugate J_\Phi) \widetilde{  u} \in H(\div;Q)$ and
    \begin{align}\label{eqn:div_composition}
        \div_Q{( (\adjugate J_\Phi) \widetilde{  u})} 
        = (\det{J_\Phi})\widetilde{\div_\Omega   u}.
    \end{align}
\end{prop}
\begin{proof}
    Set $J:=J_\Phi =  \Grad(\Phi)$ and $j:=\det(J)$.
    
    Let $  u \in H(\div;\Omega)$ and $v \in C_c^{0,1}(\Omega) \subseteq H_0^1(\Omega) = D(\gradcirc_\Omega)$. By a change of variables, we have
    \begin{align}
        \langle -\div_\Omega   u, v \rangle_{L^2(\Omega)}
        = \langle j(\widetilde{-\div_\Omega   u}), \widetilde{v} \rangle_{L^2(Q)}. \label{eqn:div_composition_calc1}
    \end{align}
    On the other hand, we may also compute
    \begin{alignat}{3}
        \langle -\div_\Omega   u, v \rangle_{L^2(\Omega)}
        &= \langle   u, \gradcirc_\Omega v \rangle_{L^2(\Omega)^d} \nonumber\\
        &= \langle j \widetilde{  u}, \widetilde{\gradcirc_\Omega v} \rangle_{L^2(Q)^d} \qquad
        &&\text{(by a change of variables)} \nonumber\\
        &= \langle j J^{-1} \widetilde{  u}, J^\top\widetilde{\gradcirc_\Omega v} \rangle_{L^2(Q)^d}\nonumber\\
        &= \langle (\adjugate J) \widetilde{  u}, \gradcirc_Q \widetilde{v} \rangle_{L^2(Q)^d} \quad
        &&\text{(by Proposition \ref{prop:grad_composition}).} \label{eqn:div_composition_calc2}
    \end{alignat}
    Since $v \in C_c^{0,1}(\Omega)$ is arbitrary, the right-hand sides of \eqref{eqn:div_composition_calc1} and \eqref{eqn:div_composition_calc2} show that $(\adjugate J_\Phi) \widetilde{  u} \in D((\gradcirc_Q)^\ast) = D(\div_Q) = H(\div;Q)$, and \eqref{eqn:div_composition} holds.
\end{proof}

The combination of Propositions \ref{prop:rot_composition} and \ref{prop:div_composition} yields Proposition \ref{prop:rot_div_comp}.

\bibliographystyle{plain}
\bibliography{ref_LSW}

@article {ABD98,
    AUTHOR = {Amrouche, C. and Bernardi, C. and Dauge, M. and Girault, V.},
     TITLE = {Vector potentials in three-dimensional non-smooth domains},
   JOURNAL = {Math. Methods Appl. Sci.},
  FJOURNAL = {Mathematical Methods in the Applied Sciences},
    VOLUME = {21},
      YEAR = {1998},
    NUMBER = {9},
     PAGES = {823--864},
      ISSN = {0170-4214,1099-1476},
   MRCLASS = {35J25 (35Q30 65N30 76D07)},
  MRNUMBER = {1626990},
MRREVIEWER = {Juha\ H.\ Videman},
       DOI =
              {10.1002/(SICI)1099-1476(199806)21:9<823::AID-MMA976>3.0.CO;2-B},
       URL =
              {https://doi.org/10.1002/(SICI)1099-1476(199806)21:9<823::AID-MMA976>3.0.CO;2-B},
}

@article {BPS16,
    AUTHOR = {Bauer, S. and Pauly, D. and Schomburg, M.},
     TITLE = {The {M}axwell compactness property in bounded weak {L}ipschitz
              domains with mixed boundary conditions},
   JOURNAL = {SIAM J. Math. Anal.},
  FJOURNAL = {SIAM Journal on Mathematical Analysis},
    VOLUME = {48},
      YEAR = {2016},
    NUMBER = {4},
     PAGES = {2912--2943},
      ISSN = {0036-1410,1095-7154},
   MRCLASS = {35Q61 (35A23)},
  MRNUMBER = {3542004},
MRREVIEWER = {Marcus\ Waurick},
       DOI = {10.1137/16M1065951},
       URL = {https://doi.org/10.1137/16M1065951},
}

@book {cioranescu_donato,
    AUTHOR = {Cioranescu, D. and Donato, P.},
     TITLE = {An introduction to homogenization},
    SERIES = {Oxford Lecture Series in Mathematics and its Applications},
    VOLUME = {17},
 PUBLISHER = {The Clarendon Press, Oxford University Press, New York},
      YEAR = {1999},
     PAGES = {x+262},
      ISBN = {0-19-856554-2},
   MRCLASS = {35B27 (35J20 35K05 35L05 49J45 74-02 74Qxx 78M40)},
  MRNUMBER = {1765047},
MRREVIEWER = {Eugenia\ P\'erez},
}

@article {Cos91,
    AUTHOR = {Costabel, M.},
     TITLE = {A coercive bilinear form for {M}axwell's equations},
   JOURNAL = {J. Math. Anal. Appl.},
  FJOURNAL = {Journal of Mathematical Analysis and Applications},
    VOLUME = {157},
      YEAR = {1991},
    NUMBER = {2},
     PAGES = {527--541},
      ISSN = {0022-247X,1096-0813},
   MRCLASS = {35Q60 (78A25)},
  MRNUMBER = {1112332},
       DOI = {10.1016/0022-247X(91)90104-8},
       URL = {https://doi.org/10.1016/0022-247X(91)90104-8},
}

@article {GSS21,
    AUTHOR = {Gallistl, D. and Sprekeler, T. and S\"{u}li, E.},
     TITLE = {Mixed {F}inite {E}lement {A}pproximation of {P}eriodic
              {H}amilton--{J}acobi--{B}ellman {P}roblems {W}ith
              {A}pplication to {N}umerical {H}omogenization},
   JOURNAL = {Multiscale Model. Simul.},
  FJOURNAL = {Multiscale Modeling \& Simulation. A SIAM Interdisciplinary
              Journal},
    VOLUME = {19},
      YEAR = {2021},
    NUMBER = {2},
     PAGES = {1041--1065},
      ISSN = {1540-3459},
   MRCLASS = {35B27 (35J60 65N12 65N15 65N30)},
  MRNUMBER = {4272900},
       DOI = {10.1137/20M1371397},
       URL = {https://doi.org/10.1137/20M1371397},
}

@article {GS19,
    AUTHOR = {Gallistl, D. and S\"{u}li, E.},
     TITLE = {Mixed finite element approximation of the
              {H}amilton-{J}acobi-{B}ellman equation with {C}ordes
              coefficients},
   JOURNAL = {SIAM J. Numer. Anal.},
  FJOURNAL = {SIAM Journal on Numerical Analysis},
    VOLUME = {57},
      YEAR = {2019},
    NUMBER = {2},
     PAGES = {592--614},
      ISSN = {0036-1429},
   MRCLASS = {65N30 (65N12 65N15 65N50)},
  MRNUMBER = {3924618},
       DOI = {10.1137/18M1192299},
       URL = {https://doi.org/10.1137/18M1192299},
}

@book {Gri85,
    AUTHOR = {Grisvard, P.},
     TITLE = {Elliptic problems in nonsmooth domains},
    SERIES = {Monographs and Studies in Mathematics},
    VOLUME = {24},
 PUBLISHER = {Pitman (Advanced Publishing Program), Boston, MA},
      YEAR = {1985},
     PAGES = {xiv+410},
      ISBN = {0-273-08647-2},
   MRCLASS = {35J25 (35-02)},
  MRNUMBER = {775683},
MRREVIEWER = {P.\ Szeptycki},
}

@misc{HMS26,
Author = {Hauck, M. and Maier, R. and Sprekeler, T.},
Title = {A post-processed higher-order multiscale method for nondivergence-form elliptic equations},
Year = {2026},
Note = {arXiv:2604.15144},
Eprint = {arXiv:2604.15144},
}

@book {Leis1986,
    AUTHOR = {Leis, R.},
     TITLE = {Initial-boundary value problems in mathematical physics},
 PUBLISHER = {B. G. Teubner, Stuttgart; John Wiley \& Sons, Ltd.,
              Chichester},
      YEAR = {1986},
     PAGES = {viii+266},
      ISBN = {3-519-02102-1},
   MRCLASS = {35-02},
  MRNUMBER = {841971},
MRREVIEWER = {O.\ John},
       DOI = {10.1007/978-3-663-10649-4},
       URL = {https://doi.org/10.1007/978-3-663-10649-4},
}

@article {Mitrea2001,
    AUTHOR = {Mitrea, M.},
     TITLE = {Dirichlet integrals and {G}affney-{F}riedrichs inequalities in
              convex domains},
   JOURNAL = {Forum Math.},
  FJOURNAL = {Forum Mathematicum},
    VOLUME = {13},
      YEAR = {2001},
    NUMBER = {4},
     PAGES = {531--567},
      ISSN = {0933-7741,1435-5337},
   MRCLASS = {35B65 (35B45 35P15 46E35 46N20)},
  MRNUMBER = {1830246},
MRREVIEWER = {Florin\ Iacob},
       DOI = {10.1515/form.2001.021},
       URL = {https://doi.org/10.1515/form.2001.021},
}

@article {Pau19,
    AUTHOR = {Pauly, D.},
     TITLE = {On the {M}axwell and {F}riedrichs/{P}oincar\'e{} constants in
              {ND}},
   JOURNAL = {Math. Z.},
  FJOURNAL = {Mathematische Zeitschrift},
    VOLUME = {293},
      YEAR = {2019},
    NUMBER = {3-4},
     PAGES = {957--987},
      ISSN = {0025-5874,1432-1823},
   MRCLASS = {35Q60 (35A23 35Q61 46E40 53A45 78A25)},
  MRNUMBER = {4024573},
MRREVIEWER = {Piero\ Antonio\ D'Ancona},
       DOI = {10.1007/s00209-018-2218-7},
       URL = {https://doi.org/10.1007/s00209-018-2218-7},
}

@incollection{PaulyWaurick2026,
  author    = {Pauly, D. and Waurick, M.},
  title     = {Gaffney's Inequality and the Closed Range Property for the de Rham Complex in Unbounded Domains},
  booktitle = {PDEs, Operator Theory, and Mathematical Physics: A Tribute to the Scientific Work of Rainer Picard on the Occasion of His 80th Birthday},
  editor    = {Trostorff, Sascha and Waurick, Marcus},
  publisher = {Birkh{\"a}user},
  year      = {2026},
  url       = {https://arxiv.org/abs/2602.00581}
}

@article {Picard1984,
    AUTHOR = {Picard, R.},
     TITLE = {An elementary proof for a compact imbedding result in
              generalized electromagnetic theory},
   JOURNAL = {Math. Z.},
  FJOURNAL = {Mathematische Zeitschrift},
    VOLUME = {187},
      YEAR = {1984},
    NUMBER = {2},
     PAGES = {151--164},
      ISSN = {0025-5874,1432-1823},
   MRCLASS = {35Q20 (58G25 78A25)},
  MRNUMBER = {753428},
MRREVIEWER = {Rolf\ Leis},
       DOI = {10.1007/BF01161700},
       URL = {https://doi.org/10.1007/BF01161700},
}

@article {Sar82,
    AUTHOR = {Saranen, J.},
     TITLE = {On an inequality of {F}riedrichs},
   JOURNAL = {Math. Scand.},
  FJOURNAL = {Mathematica Scandinavica},
    VOLUME = {51},
      YEAR = {1982},
    NUMBER = {2},
     PAGES = {310--322},
      ISSN = {0025-5521,1903-1807},
   MRCLASS = {35B45 (35F15 35Q20)},
  MRNUMBER = {690534},
MRREVIEWER = {Rolf\ Leis},
       DOI = {10.7146/math.scand.a-11983},
       URL = {https://doi.org/10.7146/math.scand.a-11983},
}

@book {Sch12,
    AUTHOR = {Schm\"udgen, K.},
     TITLE = {Unbounded self-adjoint operators on {H}ilbert space},
    SERIES = {Graduate Texts in Mathematics},
    VOLUME = {265},
 PUBLISHER = {Springer, Dordrecht},
      YEAR = {2012},
     PAGES = {xx+432},
      ISBN = {978-94-007-4752-4},
   MRCLASS = {47-01 (47B25 47E05)},
  MRNUMBER = {2953553},
MRREVIEWER = {G.\ V.\ Rozenblum},
       DOI = {10.1007/978-94-007-4753-1},
       URL = {https://doi.org/10.1007/978-94-007-4753-1},
}

@article {Spr24,
    AUTHOR = {Sprekeler, T.},
     TITLE = {Homogenization of nondivergence-form elliptic equations with
              discontinuous coefficients and finite element approximation of
              the homogenized problem},
   JOURNAL = {SIAM J. Numer. Anal.},
  FJOURNAL = {SIAM Journal on Numerical Analysis},
    VOLUME = {62},
      YEAR = {2024},
    NUMBER = {2},
     PAGES = {646--666},
      ISSN = {0036-1429,1095-7170},
   MRCLASS = {65N12 (35B27 35J25 65N30)},
  MRNUMBER = {4711961},
       DOI = {10.1137/23M1580279},
       URL = {https://doi.org/10.1137/23M1580279},
}

@misc{Spr26,
Author = {Sprekeler, T.},
Title = {A {C}ordes framework for stationary {F}okker--{P}lanck--{K}olmogorov equations},
Year = {2026},
Note = {arXiv:2601.14548},
Eprint = {arXiv:2601.14548},
}

@article {SSZ25,
    AUTHOR = {Sprekeler, T. and S\"uli, E. and Zhang, Z.},
     TITLE = {Finite element approximation of stationary
              {F}okker-{P}lanck-{K}olmogorov equations with application to
              periodic numerical homogenization},
   JOURNAL = {SIAM J. Numer. Anal.},
  FJOURNAL = {SIAM Journal on Numerical Analysis},
    VOLUME = {63},
      YEAR = {2025},
    NUMBER = {3},
     PAGES = {1315--1343},
      ISSN = {0036-1429,1095-7170},
   MRCLASS = {65N30 (35B27 35J15 65N12 65N15)},
  MRNUMBER = {4921424},
       DOI = {10.1137/24M1692848},
       URL = {https://doi.org/10.1137/24M1692848},
}

@article {Waurick2025,
    AUTHOR = {Waurick, M.},
     TITLE = {Nonlocal {$H$}-convergence for topologically nontrivial
              domains},
   JOURNAL = {J. Funct. Anal.},
  FJOURNAL = {Journal of Functional Analysis},
    VOLUME = {288},
      YEAR = {2025},
    NUMBER = {3},
     PAGES = {Paper No. 110710, 49},
      ISSN = {0022-1236,1096-0783},
   MRCLASS = {35B27 (35Q61 74Q05)},
  MRNUMBER = {4815906},
MRREVIEWER = {Mustapha\ El Jarroudi},
       DOI = {10.1016/j.jfa.2024.110710},
       URL = {https://doi.org/10.1016/j.jfa.2024.110710},
}

@article {Weber1980,
    AUTHOR = {Weber, C.},
     TITLE = {A local compactness theorem for {M}axwell's equations},
   JOURNAL = {Math. Methods Appl. Sci.},
  FJOURNAL = {Mathematical Methods in the Applied Sciences},
    VOLUME = {2},
      YEAR = {1980},
    NUMBER = {1},
     PAGES = {12--25},
      ISSN = {0170-4214,1099-1476},
   MRCLASS = {78A25 (35P25 78A45)},
  MRNUMBER = {561375},
MRREVIEWER = {D.\ L.\ Colton},
       DOI = {10.1002/mma.1670020103},
       URL = {https://doi.org/10.1002/mma.1670020103},
}

@article {Weck1974,
    AUTHOR = {Weck, N.},
     TITLE = {Maxwell's boundary value problem on {R}iemannian manifolds
              with nonsmooth boundaries},
   JOURNAL = {J. Math. Anal. Appl.},
  FJOURNAL = {Journal of Mathematical Analysis and Applications},
    VOLUME = {46},
      YEAR = {1974},
     PAGES = {410--437},
      ISSN = {0022-247X},
   MRCLASS = {78.35 (53C20 58G99)},
  MRNUMBER = {343771},
MRREVIEWER = {N.\ D.\ Sengupta},
       DOI = {10.1016/0022-247X(74)90250-9},
       URL = {https://doi.org/10.1016/0022-247X(74)90250-9},
}

\end{document}